\documentclass[a4paper,11pt]{amsart}

\title[Commutators in finite $p$-groups]{Commutators in finite $p$-groups with $3$-Generator Derived Subgroup}

\usepackage{amsfonts,amssymb,mathtools}
\usepackage{tikz}
\usepackage{textcomp, color}
\usepackage{hyperref}
\usepackage{amsmath}

\theoremstyle{definition}

\newtheorem{definition}{Definition}[section]

\theoremstyle{remark}

\newtheorem{remark}[definition]{Remark}

\theoremstyle{plain}

\newtheorem*{thmA}{Theorem A}
\newtheorem*{thmB}{Theorem B}
\newtheorem*{thmAp}{Theorem A$'$}
\newtheorem*{thmBp}{Theorem B$'$}
\newtheorem{theorem}[definition]{Theorem}
\newtheorem{lemma}[definition]{Lemma}
\newtheorem{example}[definition]{Example}

\newcommand{\F}{\mathbb{F}}

\newcommand{\N}{\mathbb{N}}

\DeclareMathOperator{\Cl}{Cl}

\begin{document}

\author[I.\ de las Heras]{Iker de las Heras}
\address{Department of Mathematics\\ University of the Basque Country UPV/EHU\\
48080 Bilbao, Spain}
\email{iker.delasheras@ehu.eus}

\thanks{The author is supported by the Spanish Government grant MTM2017-86802-P and by the Basque Government grant IT974-16.
He is also supported by a predoctoral grant of the University of the Basque Country}

\begin{abstract}
    It is well known that, in general, the set of commutators of a group $G$ may not be a subgroup. Guralnick showed that if $G$ is a finite $p$-group with $p\ge 5$ such that $G'$ is abelian and 3-generator, then all the elements of the derived subgroup are commutators. In this paper, we extend Guralnick's result by showing that the condition of $G'$ to be abelian is not needed. In this way, we complete the study of this property in finite $p$-groups in terms of the number of generators of the derived subgroup. We will also see that the same result is true when the action of $G$ on $G'$ is uniserial modulo $(G')^p$ and $|G':(G')^p|$ does not exceed $p^{p-1}$. Finally, we will prove that analogous results are satisfied when working with pro-$p$ groups.
\end{abstract}

\maketitle

\section{Introduction}

Let $K(G)$ denote the set of commutators of a group $G$.
It is a well-known problem deciding whether the derived subgroup $G'=\langle K(G)\rangle$ equals $K(G)$.
Indeed, in general this equality does not hold since the product of two commutators need not be a commutator.
As an easy example, we can consider the group $G=F_d/\gamma_3(F_d)F_d^p$, where $F_d$ is the free group on $d\ge 6$ generators and $p>2$ is a prime.
It is immediate by \cite[Theorem 3.1]{KM2} that this group does not satisfy the property, even if the nilpotency class of $G$ is 2 and the exponent of $G$ is $p$.

We need, then, to restrict our choice of the group $G$ to some particular family of groups if we want it to satisfy the desired property.
For instance, Liebeck, O'Brien, Shalev and Tiep proved in \cite{LOST} that if $G$ is a finite simple group, then $G'=K(G)$, thereby proving the so-called Ore Conjecture.
On the other hand, Guralnick (\cite[Theorem 1]{gur2}) and Kappe and Morse (\cite[Theorem 3.4 and Theorem 4.2]{KM2}) found some upper bounds for the order of $G$, for the order of $G'$ and for $n$ if $G$ is a $p$-group of order $p^n$, in such a way that a group $G$ will satisfy the equality whenever it satisfies one of these bounds.

However, we will focus on restrictions on the number of generators of the derived subgroup $G'$. In this direction, Macdonald proved in \cite{mac} that even if $G'$ is cyclic the property may fail. Actually, he showed that for every $n\in\N$ we can find a group $G$ such that $G'$ is cyclic but cannot be generated by less than $n$ commutators. This shows how delicate the equality $G'=K(G)$ can be. The situation, fortunately, is much better when working with nilpotent groups. In that case, Rodney proved in \cite{rod} that if $G$ is nilpotent with cyclic derived subgroup, then $G'=K(G)$.

The study of this property for finite nilpotent groups is clearly reduced to finite $p$-groups. If $G$ is a finite $p$-group with 2-generator abelian derived subgroup, Guralnick proved in \cite[Theorem A]{gur} that $G'$ consists only of commutators. In \cite{DF} Fern\'andez-Alcober and the author extended this result, showing that the condition that $G'$ is abelian is not necessary.

\begin{theorem}[\cite{DF}, Theorem A]
\label{theorem d(G')=2}
Let $G$ be a finite $p$-group.
If $G'$ can be generated by $2$ elements, then $G'=\{[x,g] \mid g\in G\}$ for a suitable $x\in G$.
\end{theorem}

Methods and strategies developed in \cite{DF} for the proof of Theorem \ref{theorem d(G')=2} will have a great importance when proving Theorem A below. This theorem concerns finite $p$-groups with 3-generator derived subgroup. In this context Rodney addressed the simplest cases, namely, the one when the nilpotency class of $G$ is 2 (\cite[Theorem A]{rodab}) and the one when $G'$ is elementary abelian of rank 3 (\cite[Theorem B]{rodab}), showing that in both cases we have $G'=K(G)$. Notice, however, that Rodney's results involve only groups for which $G'$ is abelian. Thus, Guralnick generalized these results for $p\ge 5$, proving that if $G$ is a finite $p$-group with $G'$ abelian and 3-generator, then $G'=K(G)$ (\cite[Theorem B]{gur}).
Moreover, he found counterexamples showing that the result is false for $p=2$ or $p=3$, even if $G'$ is abelian (\cite{gur}, Example 3.5 and Example 3.6).
In Theorem A we generalize Guralnick's result to groups in which $G'$ need not be abelian.

\begin{thmA}
Let $G$ be a finite $p$-group with $p\ge 5$. If $G'$ can be generated by 3 elements, then $G'$ consists only of commutators.
\end{thmA}

In this case, as shown in Remark \ref{remark fixed element}, it is not true, in general, that there exists a fixed element $x\in G$ such that $G'=\{[x,g]\mid g\in G\}$, as we have in Theorem \ref{theorem d(G')=2} or Theorem B below.

Macdonald (\cite[Exercise 5, page 78]{mac2}) and Kappe and Morse (\cite[Example 5.4]{KM2}) showed that for every prime $p$ there exist finite $p$-groups with 4-generator abelian derived subgroup such that $G'\neq K(G)$.
These examples show that the property may fail if the derived subgroup has more than 3 generators.
Therefore, with Theorem A and Theorem \ref{theorem d(G')=2}, we close the gap between the case when $G'$ is abelian and can be generated by 3 elements and the case when $G'$ is generated by more than 3 elements. Thus, the study of the condition $G'=K(G)$ in terms of the number of generators of the derived subgroup is complete for finite $p$-groups.
In Theorem B we show that with some additional restriction, groups with $d(G')\ge 4$ satisfy the desired equality.

\begin{thmB}
\label{theorem uniserial action}
Let $G$ be a finite $p$-group and write $d=\log_p|G':(G')^p|$. If $d\le p-1$ and the action of $G$ on $G'$ is uniserial modulo $(G')^p$, then there exists $x\in G$ such that $G'=\{[x,g]\mid g\in G\}$.
\end{thmB}

More information about the condition $G'=K(G)$ can be found in the papers \cite{KM1} and \cite{KM2}.

Finally, we show that analogous results to Theorem A and Theorem B are satisfied when working with pro-$p$ groups. Recall that if a pro-$p$ group $G$ is topologically finitely generated, then the index of $G^p$ in $G$ is a $p$-power (and in particular finite).

\begin{thmAp}
\label{theorem d(G')=3 pro}
Let $G$ be a pro-$p$ group with $p\ge 5$. If $G'$ can be topologically generated by 3 elements, then $G'$ consists only of commutators.
\end{thmAp}

\begin{thmBp}
\label{theorem uniserial action pro}
Let $G$ be pro-$p$ group with topologically finitely generated derived subgroup. Write $d=\log_p|G':(G')^p|$. If $d\le p-1$ and the action of $G$ on $G'$ is uniserial modulo $(G')^p$, then there exists $x\in G$ such that $G'=\{[x,g]\mid g\in G\}$.
\end{thmBp}

\textit{Notation and organization.\/}
Let $G$ be a group, and let $H\le G$.
We write $H\max G$ to denote that $H$ is maximal in $G$.
If $x\in G$, then we set $K_x(H)=\{ [x,h]\mid h\in H \}$ and $[x,H]=\langle K_x(H) \rangle$.
We denote the Frattini subgroup of $G$ by $\Phi(G)$.
If $G$ is finitely generated, $d(G)$ stands for the minimum number of generators of $G$.
Finally, if $G$ is a topological group and $H\le G$, we write $\Cl_G(H)$ to refer to the topological closure of $H$ in $G$ and we write $H\trianglelefteq_{\mathrm{o}} G$ to denote that $H$ is an open normal subgroup of $G$.

We start with some preliminary results in Section \ref{section preliminaries}.
Theorem B will be used in the proof of Theorem A, so it will be proved before Theorem A in Section \ref{section uniserial case}. We then split the proof of Theorem A into two sections, dealing separately with the following two cases: in Section \ref{section powerful case} we prove the result when $G'$ is powerful and in Section \ref{section non-powerful case} we prove it in the general case. Finally, in Section \ref{section pro case} we prove Theorem A$'$ and Theorem B$'$.

\section{Preliminary Results}\label{section preliminaries}

In the proof of Theorem \ref{theorem d(G')=2} (\cite[Theorem A]{DF}), the authors rely on a result by Blackburn, according to which $G'$ is powerful whenever $d(G')\le 2$ (\cite[Theorem 1]{black}).
In this way, they reduce the proof to the case in which $G'$ is powerful.
Unfortunately, this is not true when $d(G')=3$, as Example \ref{example huppert} below shows.
However, we will see in Section \ref{section non-powerful case} that the groups in which $d(G')=3$ but $G'$ is non-powerful are very specific.
Powerful groups, then, will be essential in this paper.
Background on such groups can be found in \cite[Chapter 2]{DDMS} or \cite[Chapter 11]{khu}.
These groups are usually seen as a generalization of abelian groups since they satisfy, among others, the following properties:
\begin{enumerate}
    \item $\Phi(G)=G^p$. In particular $|G:G^p|=p^{d(G)}$.
    \item $d(H)\le d(G)$ for every $H\le G$.
    \item $G^p=\{g^p\mid g\in G\}$.
    \item If $G=\langle x_1,\ldots,x_n\rangle$, then $G^p=\langle x_1^p,\ldots,x_n^p\rangle$.
    \item The power map from $G^{p^{i-1}}/G^{p^i}$ to $G^{p^i}/G^{p^{i+1}}$ that sends $gG^{p^i}$ to $g^{p}G^{p^{i+1}}$ is an epimorphism for every $i\ge 0$.
\end{enumerate}

\begin{remark}
\label{remark index}
Property (v) implies that if $G^p\le N\le L\le G$, then 
$$|L^{p^i}:N^{p^i}|\le|L:N|$$
(and hence $|N:N^{p^i}|\le |L:L^{p^i}|$), and if $L/N=\langle x_1,\ldots,x_n\rangle N$, then
$$L^{p^i}/N^{p^i}=\langle x_1^{p^i},\ldots,x_n^{p^i}\rangle N^{p^i}.$$
\end{remark}

We can generalize this concept even more with the notion of potent $p$-groups, which will also have an important role in the paper. For instance, as we will see in the proof of Theorem B, if a group satisfies the conditions of the theorem, then its derived subgroup is potent. A finite $p$-group $G$ is said to be potent if $\gamma_{p-1}(G)\le G^p$ for odd $p$ or if $G'\le G^4$ for $p=2$.
In this context, the following lemma, which is a reduced version of a theorem by Gonz\'alez-S\'anchez and Jaikin-Zapirain, will be particularly helpful. First, recall that a group $G$ is said to be power abelian if it satisfies the following three properties for all $i\ge 0$:
\begin{enumerate}
    \item $G^{p^i}=\{g^{p^i}\mid g\in G\}$.
    \item $\Omega_i(G)=\{g\in G\mid o(g)\le p^i\}$.
    \item $|G:G^{p^i}|=|\Omega_i(G)|$.
\end{enumerate}

\begin{lemma}[\cite{potent}, Theorem 1.1]
\label{lemma potent}
Let $G$ be a potent $p$-group with $p>2$. Then:
\begin{enumerate}
    \item If $N\trianglelefteq G$ then $N$ is power abelian.
    \item If $N\le G^p$ and $N\trianglelefteq G$, then $N$ is powerful.
\end{enumerate}
\end{lemma}



\vspace{2pt}

Following the strategy developed in \cite{DF}, the next two lemmas will be crucial.
Lemma \ref{lemma domino} says that if we want to show that a subgroup contains only commutators with a fixed element in the first position, we only have to care about the factors of a normal series.
Lemma \ref{lemma central} shows that, actually, it suffices to find some suitable generators for such factors.

\begin{lemma}[\cite{DF}, Lemma 2.3]
\label{lemma domino}
Let $G$ be a group and let $N\le L\le G$, with $N$ normal in $G$.
Suppose that for some $x\in G$ the following two conditions hold:
\begin{enumerate}
\item
$L/N\subseteq K_{xN}(G/N)$.
\item
$N\subseteq K_x(G)$.
\end{enumerate}
Then $L\subseteq K_x(G)$.
\end{lemma}

\begin{lemma}[\cite{DF}, Lemma 2.4]
\label{lemma central}
Let $G$ be a group and let $N\le L\le G$, with $N$ normal in $G$.
If $L/N=\langle [x,s]N \mid s\in S \rangle$ for some $x\in G$ and some $S\subseteq G$ with $[L,S]\subseteq N$,
then $L/N\subseteq K_{xN}(\langle S \rangle N/N)\subseteq K_{xN}(G/N)$.
\end{lemma}

In order to apply these lemmas we will use the following result.

\begin{lemma}
\label{lemma hall}
Let $G$ be a finite $p$-group with $p\ge 3$ and $(G')^{p^k}$ powerful for some $k\ge 0$, and let $L,N$ be two normal subgroups of $G$ such that $((G')^{p^k})^p\le N\le L\le (G')^{p^k}$.
Write $d=d((G')^{p^k})$ and suppose $d\le p^{k+1}-p^k-1$.
If $L/N=\langle [x,g]N\rangle$ where $x\in G$ and $g\in G^{p^k}$, then
$$[x,g]^{p^i}\equiv [x,g^{p^i}]\pmod{N^{p^{i}}},$$
and $L^{p^{i}}/N^{p^i}=\langle [x,g^{p^i}]N^{p^i}\rangle$ for every $i\ge 0$.
\end{lemma}
\begin{proof}
We will argue by induction on $i$. If $i=0$ there is nothing to prove, so assume $i\ge 1$ and suppose
$$[x,g]^{p^{i-1}}\equiv [x,g^{p^{i-1}}]\pmod{N^{p^{i-1}}}$$
and $L^{p^{i-1}}/N^{p^{i-1}}=\langle [x,g^{p^{i-1}}]N^{p^{i-1}}\rangle$.
By Lemma \ref{lemma potent}, $L$ and $N$ are power abelian, so $(L^{p^{i-1}})^p=L^{p^i}$ and $(N^{p^{i-1}})^p=N^{p^i}$.
Since $((G')^{p^k})^{p^{i-1}}$ is powerful, Remark \ref{remark index} yields
$$L^{p^{i}}/N^{p^i}=\langle [x,g^{p^{i-1}}]^pN^{p^i}\rangle.$$
Thus, we only have to prove that
$$[x,g^{p^{i-1}}]^p\equiv [x,g^{p^{i}}]\pmod{N^{p^{i}}}.$$
By the Hall-Petresco Identity,
$$[x,g^{p^{i-1}}]^p=[x,g^{p^{i}}]c_2^{\binom{p}{2}}c_3^{\binom{p}{3}}\ldots c_p,$$
where $c_j\in\gamma_j(\langle[x,g^{p^{i-1}}],g^{p^{i-1}}\rangle)\le[L^{p^{i-1}},G^{p^{k}},\overset{j-1}{\ldots},G^{p^{k}}]$ for every $2\le j\le p$. Note that $L/N$ is cyclic of exponent $p$, so $|L:N|\le p$ and by Remark \ref{remark index} we have $|L^{p^{i-1}}:N^{p^{i-1}}|\le p$, so that $[L^{p^{i-1}},G]\le N^{p^{i-1}}$. Hence, since $N$ is power abelian, if $2\le j\le p-1$ we have $c_j^{\binom{p}{j}}\in N^{p^i}$.

If $j=p$, then $c_p\in[L^{p^{i-1}},G^{p^{k}},\overset{p-1}{\ldots},G^{p^{k}}]$. Recall that $(G')^{p^k}$ is powerful, so we have $|((G')^{p^k})^{p^{i-1}}:((G')^{p^k})^{p^i}|\le p^d$, and hence $|N^{p^{i-1}}:N^{p^i}|\le p^{d}$ by Remark \ref{remark index}. If $k=0$, since $d\le p-2$, we get
$$c_p\in[L^{p^{i-1}},G,\overset{p-1}{\ldots},G]\le [L^{p^{i-1}},G,\overset{d+1}{\ldots},G]\le [N^{p^{i-1}},G,\overset{d}{\dots},G]\le N^{p^i}.$$
If $k\ge 1$, then it can be proved using again the Hall-Petresco Identity that for every normal subgroup $H$ of $G$ we have
$$[H,G^{p^k}]\le [H,G]^{p}[H,G,\overset{p^k}{\ldots},G],$$
so
\begin{equation*}
    \begin{split}
    c_p\in[L^{p^{i-1}},G^{p^k},\overset{p-1}{\ldots},G^{p^k}]&\le[L^{p^{i-1}},G]^{p}[L^{p^{i-1}},G,\overset{(p-1)p^k}{\ldots},G]\\
    &\le N^{p^i}[N^{p^{i-1}},G,\overset{(p-1)p^k-1}{\ldots},G]\le N^{p^i},
    \end{split}
\end{equation*}
where the last equality holds since $d\le p^{k+1}-p^k-1$. The result follows.
\end{proof}

Thus, combining Lemma \ref{lemma domino}, Lemma \ref{lemma central} and Lemma \ref{lemma hall} we get the following useful result.

\begin{lemma}
\label{lemma hall2}
Let $G$ be a finite $p$-group with $p\ge 3$ and $(G')^{p^k}$ powerful for some $k\ge 0$.
Write $d((G')^{p^k})=d$ and suppose $d\le p^{k+1}-p^k-1$. If there exist $x\in G$, $g_0,\ldots,g_{d-1}\in G^{p^k}$ and a series from $(G')^{p^k}$ to $((G')^{p^k})^p$
$$((G')^{p^k})^{p}=N_d<N_{d-1}<\dots<N_{0}=(G')^{p^k}$$
in which each factor $N_j/N_{j+1}$ is a chief factor of $G$ generated by the commutator $[x,g_j]N_{j+1}$, then $(G')^{p^k}=K_x(G)$.
\end{lemma}
\begin{proof}
Since $(G')^{p^k}$ is powerful we have $|(G')^{p^k}:((G')^{p^k})^p|=p^{d}$. By Remark \ref{remark index}, we have $|N_{j}^{p^i}/N_{j+1}^{p^i}|\le p$ for every $i\ge 0$, and furthermore, by Lemma \ref{lemma hall}, this quotient is generated by
$[x,g_j^{p^i}]N_{j+1}^{p^i}$ for every $i$ and $j$.
Therefore, by Lemma \ref{lemma central}, it follows that
$$N_{j}^{p^i}/N_{j+1}^{p^i}\subseteq K_{xN_{j+1}^{p^i}}(G/N_{j+1}^{p^i}).$$

Thus, we have a series from $(G')^{p^k}$ to $1$ in which all factors are chief factors of $G$ and all elements of each chief factor are images of commutators of the form $[x,g]$ with $g\in G$.
The result follows by applying Lemma \ref{lemma domino} again and again.
\end{proof}

\begin{remark}
Lemma \ref{lemma hall2} (and hence also Lemma \ref{lemma hall}) will be used with $k\neq 0$ only when proving Theorem B, where we use it with $k=1$. The general result has been proved for completeness.
\end{remark}



As in \cite{DF} the subgroups below will have an essential role in the paper.

\begin{definition}
\label{definition D}
Let $G$ be a non-abelian finite $p$-group.
For every $T\max G'$ with $T\trianglelefteq G$ we define the subgroup $D(T)$ by the condition
\[
D(T)/T = Z(G/T),
\]
that is, $D(T)$ is the largest subgroup of $G$ satisfying $[D(T),G]\le T$.
We set $D=\cup \{ D(T) \mid T\max G'\text{ with }T\trianglelefteq G \}$.
\end{definition}

\begin{definition}
Let $G$ be a finite $p$-group with $G'$ powerful.
We define $C=C_G(G'/(G')^p)$.
\end{definition}

Recall that the action of $G$ on a normal subgroup $N$ of $G$ is uniserial if
$$|[N,G,\overset{i}{\ldots},G]:[N,G,\overset{i+1}{\ldots},G]|\le p$$
for every $i\ge 0$. We also define the following subgroups, which are just the so-called two-step centralizers modulo $(G')^p$.

\begin{definition}
Let $G$ be a finite $p$-group such that the action of $G$ on $G'$ is uniserial modulo $(G')^p$. Then, we define
$$C_i=C_G(\gamma_i(G)(G')^p/\gamma_{i+2}(G)(G')^p)$$
for every $i\ge 2$ such that $\gamma_{i+1}(G)\not\le (G')^p$.
\end{definition}

\begin{remark}
\label{remark step}
In the situation above, the subgroups $C_i$ are all maximal in $G$ since $|\gamma_{i}(G)(G')^p:\gamma_{i+2}(G)(G')^p|=p^2$ and $[\gamma_i(G)(G')^p,G]\not\le\gamma_{i+2}(G)(G')^p$.
\end{remark}

We prove the following result exactly in the same way as in \cite[Lemma 2.9]{DF}, even if $d(G')\neq 2$.

\begin{lemma}
\label{lemma D}
If $G$ is a non-abelian finite $p$-group then $[x,G]=G'$ if and only  if $x\not\in D$.
Furthermore, for every $T\max G'$ with $T\trianglelefteq G$, we have $\Phi(G)\le D(T)$ and $\log_p |G:D(T)|$ is even.
\end{lemma}
\begin{proof}
Since $[x,G]$ is a normal subgroup of $G$, we have $[x,G]<G'$ if and only if $x\in D(T)$ for some $T\max G'$ with $T\trianglelefteq G$, and the first assertion follows.

On the other hand, let $T\max G'$ with $T\trianglelefteq G$.
We have $[\Phi(G),G]=[G^p,G]\gamma_3(G)=(G')^p\gamma_3(G)\le T$, and so $\Phi(G)\le D(T)$.
Thus $G/D(T)$ can be seen as an $\F_p$-vector space.
In addition, the  commutator map in $G/T$ induces a a non-degenerate alternating form on $G/D(T)$, and thus
$\dim_{\F_p} G/D(T)$ is even.
\end{proof}

\section{Proof of Theorem B}\label{section uniserial case}

Before proving Theorem B we need the following simple lemma, according to which the first part of Remark \ref{remark index} can be stated in a more general way, even if $G$ is potent.

\begin{lemma}
\label{lemma index}
Let $G$ be a potent $p$-group with $p\ge 3$. If $N\le L$ are two normal subgroups of $G$, then $|N:N^{p^i}|\le|L:L^{p^i}|$ for all $i\ge 0$. In particular $|L^{p^i}:N^{p^i}|\le |L:N|$.
\end{lemma}
\begin{proof}
By Lemma \ref{lemma potent}, the subgroups $N$ and $L$ are power abelian, so in particular $|N:N^{p^i}|=|\Omega_i(N)|$ and $|L:L^{p^i}|=|\Omega_i(L)|$. Since obviously $|\Omega_i(N)|\le|\Omega_i(L)|$, the result follows.
\end{proof}

\begin{proof}[Proof of Theorem B]If $d=1$, then $G'$ is cyclic and the result follows from Theorem \ref{theorem d(G')=2}, so assume $d\ge 2$ (and in particular $p\ge 3$).
For the sake of simplicity we will write $G_i=\gamma_i(G)(G')^p$, so that
$$(G')^p=G_{d+2}\le G_{d+1}\le \ldots \le G_3\le G_2=G'$$
is a series from $G'$ to $(G')^p$ such that $|G_i:G_{i+1}|=p$ for all $2\le i \le d+1$.
Note that if $N\max G'$ with $N\trianglelefteq G$, then $G_3\le N$.
Therefore, $N=G_3$ and $G_3$ is the unique subgroup satisfying those conditions.
Hence, $D$ is a subgroup of $G$ whose index is greater than $p$ by Lemma \ref{lemma D}.
Note also that there are only $d-1\le p-2$ two-step centralizers, which are maximal by Remark \ref{remark step}.
Thus, 
we can take $x\in G\setminus (D\cup C_2\cup \ldots\cup C_d)$.
By Lemma \ref{lemma D} we have $G'=[x,G]$ and since $C_2$ is maximal in $G$ we have $G'=[x,G]=[x,\langle x\rangle C_2]=[x,C_2]$.
In particular $G'/G_3=\langle[x,g_1]G_3\rangle$ for some $g_1\in C_2$.
Furthermore, since $x\not\in C_i$ for $2\le i\le d$, we also have $G_{i+1}/G_{i+2}=\langle[x,g_i]G_{i+2}\rangle$ for some suitable $g_i\in G_i$.
It follows from Lemma \ref{lemma central} and Lemma \ref{lemma domino} that $G'/(G')^p\subseteq K_{x(G')^p}(G/(G')^p)$.

Recall that $\gamma_{d+2}(G)\le (G')^p$, and since $d\le p-1$, it follows that $\gamma_{p-1}(G')\le\gamma_{2(p-1)}(G)\le\gamma_{2d}(G)$. Thus, since $2d\ge d+2$ we have $\gamma_{p-1}(G')\le (G')^p$, so that $G'$ is potent. In this case the power map from $G'/(G')^p$ to $(G')^p/(G')^{p^2}$ defined above Remark \ref{remark index} need not be a homomorphism. However, we can restrict its domain and codomain in order for it to be so. We claim that the map from $G_i/G_{i+1}$ to $G_i^p/G_{i+1}^p$ sending $gG_{i+1}$ to $g^pG_{i+1}^p$ is an epimorphism for every $2\le i\le d+1$.

Take $x,y\in G_i$. By the Hall-Petresco Identity we have
$$(xy)^p=x^py^pc_2^{\binom{p}{2}}c_3^{\binom{p}{3}}\ldots c_p$$
with $c_j\in \gamma_j(G_i)$.
Obviously if $2\le j\le p-1$ then $c_j^{\binom{p}{j}}\in G_{i+1}^p$. Besides, if $j=p$, since $G_i\le G'$, we have
$$c_p\in[G_i,\overset{p}{\ldots},G_i]\le[G_i,G,\overset{2(p-1)}{\ldots},G]\le G_{i+1}^p,$$
where the last inequality holds since by Lemma \ref{lemma index} we have
$$|G_i:G_{i+1}^p|=|G_i:G_i^p||G_i^p:G_{i+1}^p|\le p^{d+1}$$
and $d+1\le 2(p-1)$.
Moreover, since $G'$ is potent it follows that $G_i$ is power abelian, so the map must be an epimorphism. The claim is proved.

Thus, by Lemma \ref{lemma index} it follows that we have a series
$$((G')^p)^p=G_{d+2}^p\le G_{d+1}^p\le\ldots\le G_3^p\le G_2^p=(G')^p$$
in which each factor $G_{i+1}^p/G_{i+2}^p$ has order less than or equal to $p$ and is generated by the image of $[x,g_i]^p$ for every $1\le i\le d$.
In order to apply Lemma \ref{lemma hall2} let us prove that
$$[x,g_i]^p\equiv [x,g_i^p]\pmod{G_{i+2}^p}$$
for every $i$. Assume first $i=1$. We will use again the Hall-Petresco Identity so that
$$[x,g_1]^p=[x,g_1^p]c_2^{\binom{p}{2}}c_3^{\binom{p}{3}}\ldots c_p$$
with $c_j\in\gamma_j(\langle[x,g_1],g_1\rangle)\le[G,C_2,\overset{j}{\ldots},C_2]$.
If $2\le j\le p-1$ then $c_j^{\binom{p}{j}}\in G_3^p$.
If $j=p$, we have
$$c_p\in[G,C_2,\overset{p}{\ldots},C_2]\le[G_4,C_2,\overset{p-2}{\ldots},C_2].$$
Lemma \ref{lemma index} yields $|G_4:G_3^p|\le p^{d-1}$, and since $d-1\le p-2$, we conclude $c_p\in G_3^p$. For $i\ge 2$ we have $g_i\in G'$, so the claim follows more easily applying the same method.

Now, $d\le p-1\le p^2-p-1$, so we apply Lemma \ref{lemma hall2} with $j=1$ and we get $(G')^p\subseteq K_x(G)$. Since $G'/(G')^p\subseteq K_{x(G')^p}(G/(G')^p)$, we conclude by Lemma \ref{lemma domino}.
\end{proof}

\begin{remark}
If the exponent of $G'$ is $p$, that is, if $(G')^p=1$, then, following the same method, Theorem B can be stated for $d\le p+1$. Indeed, if $G$ is the union of $p+1$ proper subgroups, then all of them must be maximal.
\end{remark}

\section{Proof of Theorem A when $G'$ is Powerful}\label{section powerful case}

In order to prove Theorem A we need the following technical lemma, which will be very helpful when using induction on the order of the group.

\begin{lemma}
\label{lemma union}
Let $G$ be a finite $p$-group with $p\ge 5$, $G'$ powerful and $d(G')=3$. Assume there exist $x,u,v\in G$ such that $G'=\langle [u,v],[x,u],[x,v]\rangle$, $G'\neq [x,G]$ and $[x,G,G]\le(G')^p$. Then, there exists a family of proper subgroups of $G$ such that $[x,G](G')^p$ equals the union of their derived subgroups. Moreover, each of these derived subgroups is powerful.
\end{lemma}
\begin{proof}
Consider the subgroups $H_{i}=\langle x,uv^i,v^p\rangle$ for $0\le i\le p-1$ and $H_{p}=\langle x,v,u^p\rangle$.
Let us prove that $H_{i}'=\langle[x,uv^i]\rangle(G')^p$ for $0\le i\le p-1$ and that $H_p'=\langle [x,v]\rangle(G')^p$.

Suppose first $i\neq p$.
Since $G'=\langle [u,v],[x,u],[x,v]\rangle$ and $G'\neq[x,G]$, we have $|G':[x,G](G')^p|=p$, and since $[x,G,G]\le (G')^p$, the map
\begin{equation*}
    \begin{array}{ccc}
        G &  \longrightarrow  &  [x,G](G')^p/(G')^p \\
        g     &  \longmapsto  &  [x,g](G')^p
    \end{array}
\end{equation*}
is a homomorphism.
Therefore, we can write
$$G'=\langle [u,v],[x,uv^i],[x,v]\rangle.$$
Thus, since $G'$ is powerful, we have
$$(G')^p=\langle [u,v]^p,[x,uv^i]^p,[x,v]^p\rangle.$$
The subgroups $[x,G](G')^p$ and $\langle[x,v]\rangle(G')^p$ are normal in $G$ since $[x,G,G]\le (G')^p$, so taking $k=0$ in Lemma \ref{lemma hall}, it follows that
$$[u,v]^p\equiv [uv^i,v^p]\pmod{([x,G](G')^p)^p}$$
and
$$[x,v]^p\equiv[x,v^p]\pmod{(G')^{p^2}}.$$
Hence,
$$(G')^p=\langle[uv^i,v^p],[x,v^p],[x,uv^i]^p\rangle\le H_{i}',$$
so that $\langle[x,uv^i]\rangle(G')^p= H_{i}'$, as asserted.
Similar arguments imply that $H_{p}'=\langle [x,v]\rangle(G')^p$. 

It is easy to see now that $[x,G](G')^p=\bigcup_{i=0}^p H_{i}'$ (just observe that the $H_{i}'$ are precisely the subgroups between $[x,G](G')^p$ and $(G')^p$). Finally, notice that $|H_{i}':(G')^p|=p$ for every $i$, so since $(G')^p$ is powerfully embedded in $G'$, it follows by \cite[Lemma 11.7]{khu} that $H_{i}'$ is powerful. Thus, the proof is complete.
\end{proof}

We are now in a position to prove Theorem A in the case that $G'$ is powerful.

\begin{theorem}
\label{theorem d(G')=3, powerful}
Let $G$ be a finite $p$-group with $G'$ powerful, $d(G')\le 3$ and $p\ge 5$. Then, $G'=K(G)$.
\end{theorem}
\begin{proof}
We proceed by induction on the order of $G$. For $d(G')\le 2$ the result follows from Theorem \ref{theorem d(G')=2}. Now assume that $d(G')=3$ and note that we have $|G':(G')^p|=p^3$.
We will consider three different cases depending on the position of the subgroup $\Gamma=(G')^p\gamma_3(G)$.

\vspace{15pt}

\textbf{\textit{Case 1}}. $|G':\Gamma|=p$.

\vspace{15pt}




If $|\Gamma:\gamma_4(G)(G')^p|=p$, then the action of $G$ on $G'$ is uniserial modulo $(G')^p$ and the result follows from Theorem B.

Assume then $\gamma_4(G)\le (G')^p$. If $G'=K_x(G)$ for some $x\in G$, then, of course, we are done, so assume $G'\neq K_x(G)$ for every $x\in G$. We claim that there exist $u,v\in G$ such that $G'=\langle [u,v],[u,v,u],[u,v,v]\rangle$. For that purpose we can suppose that $(G')^p=1$.
As seen in the proof of Theorem B we have $D=D(\Gamma)$, and since $C$ is also a proper subgroup of $G$ (otherwise $\gamma_3(G)\le (G')^p$), we can take $u\not\in C\cup D$.
Then, $G'=[u,G]$ by Lemma \ref{lemma D} and $C_{G/\Gamma}(u\Gamma)\neq G/\Gamma$. Let us write $C^*/\Gamma=C_{G/\Gamma}(u\Gamma)$.

Since $u\not\in C$ we have $[u,G']\neq 1$. If $[u,G']=\Gamma$, then, we can find a series of normal subgroups of $G$ from $G'$ to $(G')^p$ such that all factors have order $p$ and are generated by images of elements of the form $[u,g]$ for some suitable $g\in G$.
Thus, Lemma \ref{lemma hall2} implies $G'=K_u(G)$, which is a contradiction.
Therefore, we have $|[u,G']|=p$ and hence $C_G(G'/[u,G'])\neq G$.
Take thus $v\in G\setminus (C_G(G'/[u,G'])\cup C^*)$.
Then, $G'/\Gamma=\langle[u,v]\Gamma\rangle$ (because $v\not\in C^*)$, and again, as we have seen for $u$, we also have $|[v,G']|=p$.
It follows that $[u,v,u],[u,v,v]\neq 1$.
Furthermore, since $v\not\in C_G(G'/[u,G'])$, we have $[u,G']\neq [v,G']$, and we conclude that $G'=\langle[u,v],[u,v,u],[u,v,v]\rangle$. This proves the claim.

Remove now the assumption of $(G')^p= 1$ and observe that $[[u,v],G,G]\le(G')^p$, so we are in the situation of Lemma \ref{lemma union}. It follows then that $\Gamma$ is the union of the derived subgroups of some proper subgroups of $G$. These derived subgroups are all powerful, and since $d(G')=3$, they all can be generated by $3$ elements. So, by induction, $\Gamma\subseteq K(G)$.

Take now $g\in G'\setminus \Gamma$ arbitrary. We claim that $g$ is a commutator modulo $\Gamma^{p^i}$ for every $i\ge 0$ (and hence that $g$ is a commutator). We proceed by induction on $i$. Clearly, we have $g=[x,y]z$ for some $x,y\in G$, $z\in \Gamma$, so the case $i=0$ is satisfied. Assume then that $i\ge1$ and $g=[x,y]z_1$ where $x,y\in G$ and $z_1\in\Gamma^{p^{i-1}}$.

Note that $G'/\Gamma=\langle[x,y]\Gamma\rangle$, so since $[\Gamma,G]\le(G')^p$, we have $\Gamma/(G')^p=\{[x,y,h](G')^p\mid h\in G\}$.
Besides, since $G'$ is powerful, the power map from $G'/(G')^p$ to $(G')^{p^{i-1}}/(G')^{p^i}$ is an epimorphism, so that $\Gamma^{p^{i-1}}/(G')^{p^{i}}=\{[x,y,h]^{p^{i-1}}(G')^{p^{i}}\mid h\in G\}$.
By Lemma \ref{lemma hall} we have
$$\Gamma^{p^{i-1}}/(G')^{p^{i}}=\{[x,y,h^{p^{i-1}}](G')^{p^{i}}\mid h\in G\}.$$
Thus,
$$g=[x,y][x,y,h^*]z_2=[x^{h^*},y^{h^*}]z_2$$
for some $h^*\in G$ and $z_2\in (G')^{p^{i}}$. We rewrite, in order to simplify the notation, $x$ instead of $x^{h^*}$ and $y$ instead of $y^{h^*}$, so that $g=[x,y]z_2$.

Note again that $G'/\Gamma=\langle[x,y]\Gamma\rangle$, so it follows that
$$(G')^{p^i}/\Gamma^{p^i}=\langle [x,y]^{p^i}\Gamma^{p^i}\rangle.$$
Therefore, 
$$g=[x,y][x,y]^{jp^{i}}z_3=[x,y]^{1+jp^{i}}z_3$$
with $j\ge 0$ and $z_3\in \Gamma^{p^{i}}$.
Now, by the last theorem in \cite{honda}, there exist $x',y'\in G$ such that $[x,y]^{1+jp^{i}}=[x',y']$, so $g=[x',y']z_3$ with $z_3\in \Gamma^{p^{i}}$, as claimed.

\vspace{15pt}

\textbf{\textit{Case 2}}. $|G':\Gamma|=p^2$.

\vspace{15pt}

Let us prove that $C\cup D\neq G$.
On the one hand, as seen before, $\Gamma\le N$ for all $N\max G'$ which are normal in $G$, and since $|G':\Gamma|=p^2$, there are exactly $p+1$ subgroups between $G'$ and $\Gamma$. Furthermore, since they are central over $\Gamma$, they are all normal in $G$. Thus, $D=D(U_1)\cup\ldots\cup D(U_{p+1})$, where $U_1,\ldots,U_{p+1}\max G'$ are these normal subgroups. In addition, it follows from Lemma \ref{lemma D} that $|G:D(U_i)|\ge p^2$ for every $i$.

On the other hand, observe again that $C\neq G$. Hence, if we write $|G|=p^n$, we have
\begin{align*}
    |C\cup D|&\le |C|+|D|\le \sum_{i=1}^{p+1}|D(U_i)|+|C|\\
    &\le(p+1)p^{n-2}+p^{n-1}=2p^{n-1}+p^{n-2}<p^n,
\end{align*}
as we wanted. Take now $x\not\in C\cup D$. Since $x\not\in D$ we have $G'/\Gamma=[x,G]\Gamma/\Gamma$ by Lemma \ref{lemma D}, and since $x\not\in C$ we have $\Gamma/(G')^p=[x,G'](G')^p/(G')^p$.
Thus, since all subgroups between $G'$ and $\Gamma$ are central and hence normal in $G$, we can construct a series from $G'$ to $(G')^p$ where all factors have order $p$ and are generated by images of commutators of the form $[x,g]$ with $g\in G$. Again, the result follows from Lemma \ref{lemma hall2}.

\vspace{15pt}

\textbf{\textit{Case 3}}. $\gamma_3(G)\le (G')^p$.

\vspace{15pt}

If $G'=[x,G](G')^p$ for some $x$, again, all the subgroups between $G'$ and $(G')^p$ are normal in $G$, so we could construct a series from $G'$ to $(G')^p$ in such a way that we would be done by Lemma \ref{lemma hall2}.
Therefore, assume $[x,G](G')^p<G'$ for every $x\in G$.
By \cite[Theorem B]{rodab} the result is satisfied for $G/(G')^p$, so we have $G'=\bigcup_{x\in G}[x,G](G')^p$.
Thus, it suffices to prove that $[x,G](G')^p\subseteq K(G)$ for every $x\in G$.

Suppose first $|[x,G](G')^p:(G')^p|=p$.
We claim that there always exists $y\in G$ such that $[x,G](G')^p\le [y,G](G')^p\max G'$.
For that purpose, we assume $(G')^p=1$.
Note that $|G'/[x,G]|=p^2$, so by Theorem \ref{theorem d(G')=2}, there exists $u\in G$ such that $G'/[x,G]=[u,G][x,G]/[x,G]$.
Hence $G'=[u,G][x,G]$ with $|[u,G]|=p^2$.
Observe that $C_G(u),C_G(x)\neq G$, so take $y\not\in C_G(u)\cup C_G(x)$.
Thus, $[x,G]=\langle[x,y]\rangle$, and $[u,y]\ne 1$.
If $[u,y]\in\langle[x,y]\rangle$, then $[x,y]\in[u,G]$, a contradiction.
Observe, however, that $[x,y],[u,y]\in[y,G]$, so $|[y,G]|=p^2$.
Since $[x,G]\le [y,G]$, the claim is proved.

Hence, we only have to consider the case $|[x,G](G')^p:(G')^p|=p^2$.
We claim now that there exist $u,v\in G$ such that $G'=\langle [u,v],[x,v],[x,u]\rangle$.
Assume again that $(G')^p=1$.
Since $|[x,G]|=|\{[x,g]\mid g\in G\}|=p^2$, we have $|G:C_G(x)|=p^2$, and we can consider a maximal subgroup $M$ such that $C_G(x)<M<G$.
Observe that $G'=[G,G]=[G,M]$, $G=\langle G\setminus M\rangle$ and $M=\langle M\setminus C_G(x)\rangle$.
Hence, there exist $u\in G\setminus M$ and $v\in M\setminus C_G(x)$ such that $[u,v]\not\in [x,G]$.
Furthermore, $[x,G]=\langle [x,u],[x,v]\rangle$, so $G'=\langle [u,v],[x,u],[x,v]\rangle$, as claimed.

Remove now the assumption of $(G')^p=1$ and note that we are in the situation of Lemma \ref{lemma union} since $[x,G,G]\le\gamma_3(G)\le (G')^p$. Hence we have $[x,G](G')^p\subseteq K(G)$, as we wanted.
\end{proof}

\begin{remark}
Case 2 can be generalized for $p\ge 3$ using a slightly different version of Lemma \ref{lemma hall2}, but one must be more selective in the choice of $x$.
\end{remark}

\begin{remark}
\label{remark fixed element}
It is not true that, in general, if $d(G')=3$ we have $G'=K_x(G)$ for some $x\in G$, as we had in Theorem \ref{theorem d(G')=2} or Theorem B. Indeed, let $G=F_3/\gamma_3(F_3)F_3^p$, where $F_3$ is the free group on 3 generators and $p\ge 3$ is a prime. Note that $G'$ is 3-generator and $|G:Z(G)|=p^3$. Now, if $x\in Z(G)$, then $K_x(G)=1$, and if $x\not\in Z(G)$, then $|K_x(G)|=|G:C_G(x)|\le p^2$ since $\langle Z(G),x\rangle\le C_G(x)$.
\end{remark}

\section{Proof of Theorem A when $G'$ is Non-Powerful}\label{section non-powerful case}

The following example, taken directly from \cite[Example 14.24, page 376]{huppert}, shows that unlike the case when $d(G')\le 2$, it may happen that $G'$ is non-powerful when $d(G')=3$.

\begin{example}
\label{example huppert}
Let $p\ge 5$ and consider the groups $A=\langle a_1\rangle\times\langle a_2\rangle\times\langle a_3\rangle\cong C_p\times C_p\times C_p$ and $B=\langle b_1\rangle\times\langle b_2\rangle\cong C_p\times C_p.$ Define $Y=A\rtimes B$ via the automorphisms
\begin{equation*}
    \begin{split}
    a_1^{b_1}=a_1a_3^{-1},\ \ \ a_2^{b_1}=a_2a_3,\ \ \ &a_3^{b_1}=a_3,\\
    a_1^{b_2}=a_1a_3^{-1},\ \ \ a_2^{b_2}=a_2,\ \ \ \ \ &a_3^{b_2}=a_3.
    \end{split}
\end{equation*}
Now, consider $X=\langle x\rangle\cong C_p$ and define $G=Y\rtimes X$ via the automorphism
\begin{equation*}
    \begin{split}
    a_1^{x}=a_1a_2^{-1},\ \ \ a_2^{x}&=a_2,\ \ \ a_3^{x}=a_3,\\
    b_1^{x}=b_1b_2^{-1},\ \ &\ b_2^{x}=b_2a_1^{-1}.
    \end{split}
\end{equation*}
The group $G$ is a $p$-group of maximal class of order $p^6$ and exponent $p$ such that $d(G')=3$ and $G''=\gamma_5(G)\neq 1$.
\end{example}

We will start, hence, analyzing which kind of groups may arise when $G'$ is non-powerful. Actually, we will see that in such a case, $G/(G')^p$ must be a very special kind of $p$-group, namely, a CF$(m,p)$-group. These groups were introduced by Blackburn in \cite{black2} and are defined as follows.

\begin{definition}
Let $m\in\N$ such that $m\ge 3$. A $p$-group $G$ is said to be a CF$(m,p)$-\emph{group} if the nilpotency class of $G$ is $m-1$ and the action of $G$ on $G'$ is uniserial.
\end{definition}

We next define the degree of commutativity on CF($m,p$)-groups exactly in the same way as for groups of maximal class.

\begin{definition}
Let $G$ be a CF$(m,p)$-group. The \emph{degree of commutativity} of $G$ is defined as
$$\max\{k\le m-2\mid [G_i,G_j]\le G_{i+j+k}\text{ for all }i,j\ge 1\},$$
where $G_1=C_2$ and $G_i=\gamma_i(G)$ for all $i\ge 2$.
\end{definition}

Lemma \ref{lemma 5-uniserial} below shows that if $G'$ is non-powerful, then we can reduce our proof to a very particular group which is a CF($6,p$)-group modulo $(G')^p$. The key part of the proof is the following lemma due to Blackburn.

\begin{lemma}[\cite{black2}, Theorem 2.11]
\label{lemma black2}
Let $G$ be a \emph{CF}$(m,p)$-group with $m$ odd and $5\le m\le 2p+1$. Then $G$ has degree of commutativity greater than 0.
\end{lemma}

\begin{lemma}\label{lemma 5-uniserial}
Let $G$ be a finite $p$-group with $p\ge 3$, $d(G')=3$ and $G'$ non-powerful. Then $G/(G')^p$ is a \emph{CF}$(6,p)$-group.
\end{lemma}
\begin{proof}
Clearly we can assume $(G')^p=1$ and $G''\neq 1$. Thus, the Frattini subgroup of $G'$ is $G''$, and since $d(G')=3$, then $|G':G''|=p^3$. Note that $G''\le \gamma_4(G)$, so the only possibilities for $\gamma_3(G)$ are $|G':\gamma_3(G)|= p^2$ or $|G':\gamma_3(G)|=p$.

Assume first $|G':\gamma_3(G)|=p^2$.
Then, since $G''\le\gamma_4(G)$ and $|G':G''|=p^3$ we have $G''=\gamma_4(G)$.
In addition, $G'$ has two generators modulo $\gamma_3(G)$, which implies that $|G'':\gamma_5(G)|=p$ (recall that $(G')^p=1$).
Consider the subgroup $C_3$ and recall it is maximal by Remark \ref{remark step}.
In the same way as in Case 2 of Theorem \ref{theorem d(G')=3, powerful}, it can be seen that there are only $p+1$ maximal subgroups of $G'$ that are normal in $G$. Hence, making the same computations, it follows that $D\cup C_3\neq G$.

Thus, we can pick $x\in G\setminus(D\cup C_3)$, and we have $G'=[x,G]=[x,\langle x\rangle C_3]=[x,C_3]$. We can then find $y,z\in C_3$ such that $G'=\langle [x,y],[x,z],\gamma_3(G)\rangle$. We write $a=[x,y]$ and $b=[x,z]$ for simplicity. Thus, $\gamma_4(G)=\langle [a,b],\gamma_5(G)\rangle$, and we write, again for simplicity, $d=[a,b]$.

On the one hand,
$$[b,y]^x=[b[b,x],ya^{-1}]\equiv [b,y]d\pmod{\gamma_5(G)},$$
so that $[b,y,x]\equiv d \pmod{\gamma_5(G)}$. Similarly we get
$$[z,a]^x\equiv [z,a]d\pmod{\gamma_5(G)}$$
and so $[z,a,x]\equiv d\pmod{\gamma_5(G)}$. In particular $[b,y],[z,a]\not\in\gamma_{4}(G)$, and since $\gamma_3(G)/\gamma_4(G)$ is of order $p$, we have $[z,a]\equiv[b,y]^i\pmod{\gamma_4(G)}$ for some $1\le i\le p-1$. Note, however, that $$[z,a,x]\equiv[[b,y]^i,x]\equiv[b,y,x]^i\equiv d^i\pmod{\gamma_5(G)},$$
so we get $i=1$ and thus
$$1\not\equiv[b,y]\equiv[z,a]\pmod{\gamma_4(G)}.$$

On the other hand, we have $[G',C_3']\le [G',C_3,C_3]\le \gamma_5(G)$. Denote $Z/\gamma_5(G)=Z(G'/\gamma_5(G))$. We have $|G':Z|\ge p^2$, and since $[G',\gamma_3(G)]\le \gamma_5(G)$, we get $Z=\gamma_3(G)$. In particular, we get $C_3'\le \gamma_3(G)$, and the nilpotency class of $C_3$ is less than or equal to 2. Now,
$$[y,z]^x=[ya^{-1},zb^{-1}]\equiv [y,z][b,y][z,a]d\pmod{\gamma_5(G)},$$
so that $[y,z,x]\equiv [b,y][z,a]d\pmod{\gamma_5(G)}$. This is a contradiction since $[b,y][z,a]\in\gamma_3(G)\setminus\gamma_4(G)$ but $[y,z,x],d\in\gamma_4(G)$.

Therefore we must have $|G':\gamma_3(G)|=p$.
Thus
$$G''=[G',G']=[G',\gamma_3(G)]\le \gamma_5(G),$$
and since $|G':G''|=p^3$, we have $|\gamma_3(G):\gamma_4(G)|=|\gamma_4(G):\gamma_5(G)|=p$ and $G''=\gamma_5(G)$. Let us write $\overline{G}=G/\gamma_7(G)$. Note that
$$\gamma_3(G')=[G'',G']=[\gamma_5(G),G']\le \gamma_{7}(G),$$
so $\overline{\gamma_3(G')}=\overline{1}$ and since $d(\overline{G}')=3$, then $d(\overline{G}'')\le2$. Indeed, we can write $G'=\langle a,b,c\rangle$ with $a\in G'\setminus \gamma_3(G)$, $b\in\gamma_3(G)\setminus\gamma_4(G)$ and $c\in\gamma_4(G)\setminus\gamma_5(G)$, and so the generators of $\overline{G}''$ are $\overline{[a,b]}\in\gamma_5(\overline{G})$ and $\overline{[a,c]}\in\gamma_6(\overline{G})$ (note that $\overline{[b,c]}\in\gamma_7(\overline{G})=\overline{1}$). Hence $|\gamma_5(\overline{G}):\gamma_{6}(\overline{G})|=p$ and $|\gamma_6(\overline{G})|\le p$.

If $|\gamma_6(\overline{G})|=\overline{1}$ then $\gamma_6(G)=\gamma_7(G)=1$ and we are done, so assume $|\gamma_6(\overline{G})|= p$.
Thus, $\overline{G}$ is a CF($7,p$)-group, and since $p\ge 3$, by Lemma \ref{lemma black2} it follows that the degree of commutativity of $\overline{G}$ is greater than 0. In particular we have $\overline{G}''=[\gamma_2(\overline{G}),\gamma_3(\overline{G})]\le\gamma_6(\overline{G})$, which is a contradiction. The lemma follows.
\end{proof}

With all this, the second part of the proof of Theorem A follows easily.

\begin{theorem}
\label{theorem d(G')=3, non-powerful}
Let $G$ be a finite $p$-group with $G'$ non-powerful, $d(G')\le 3$ and $p\ge 5$. Then, $G'=K(G)$.
\end{theorem}
\begin{proof} We assume $d(G')=3$ by Theorem \ref{theorem d(G')=2}. By Lemma \ref{lemma 5-uniserial} the action of $G$ on $G'$ is uniserial modulo $(G')^p$ and, in addition, $|G':(G')^p|=p^4\le p^{p-1}$ since $p\ge 5$. The result follows directly from Theorem B.
\end{proof}

Thus, combining Theorem \ref{theorem d(G')=3, powerful} and Theorem \ref{theorem d(G')=3, non-powerful} we establish Theorem A.



\section{Proof of Theorems A$'$ and B$'$}\label{section pro case}

The analogous to Theorem A and Theorem B for pro-$p$ groups can be easily proved. Let $\mathcal{A}$ and $\mathcal{B}$ be the family of groups satisfying the conditions of Theorem A and Theorem B respectively. Using this notation, we prove now Theorem A$'$ and Theorem B$'$ together.

\begin{proof}[Proof of Theorem A$'$ and Theorem B$'$]
Let $G$ be a pro-$p$ group satisfying the conditions of Theorem A$'$ or Theorem B$'$. Then, for every open normal subgroup $N$ of $G$ we have $G/N\in\mathcal{A}$ or $G/N\in\mathcal{B}$, and therefore, by Theorem A or Theorem B respectively, we have $G'N/N=K(G)N/N$ or $G'N/N=K_{x_N}(G)N/N$ for some $x_N\in G$ depending on the normal subgroup $N$.

In the case of Theorem B$'$, let $X_N=\{x\in G \mid (G/N)'=K_{xN}(G/N)\}$, which is closed in $G$, being a union of cosets of $N$.
Clearly, the family $\{ X_N \}_{N\trianglelefteq_{o} G}$ has the finite intersection property and, since $G$ is compact, $\cap_{N\trianglelefteq_{o} G} \, X_N \ne \varnothing$.
If $x$ belongs to this intersection, then $(G/N)'=K_{xN}(G/N)$ for all $N\trianglelefteq_{o} G$.
Thus, let us write $\mathcal{K}(G)$ to refer to the subset $K(G)$ if we are in the situation of Theorem A$'$, or to the subset $K_x(G)$ if we are in the situation of Theorem B$'$. Note that in both cases $\mathcal{K}(G)$ is closed in $G$, being the image of a continuous function.

On the other hand, since $G'$ is topologically finitely generated we have
$$G'=\Cl_
{G'}(\langle [x_1,x_2],\ldots,[x_{2n-1},x_{2n}]\rangle)$$
for suitable $n\ge 1$ and $x_i\in G$ with $1\le i\le 2n$.
Define $H=\Cl_G(\langle x_1,\ldots,x_{2n}\rangle)$. Since $H$ is topologically finitely generated, it follows that $H'$ is closed in $H$, and hence in $G$. Thus, $H'$ is also closed in $G'$, so $G'=H'$. Therefore, $G'$ is closed in $G$. 
Now,
$$G'=\overline{G'}=\bigcap_{N\trianglelefteq_{o}G}G'N=\bigcap_{N\trianglelefteq_{o}G}\mathcal{K}(G)N=\overline{\mathcal{K}(G)}=\mathcal{K}(G)$$
and the proof is complete.
\end{proof}

\end{document}